\documentclass[a4paper,12pt]{article}

\usepackage{amsthm}
\usepackage{amssymb}
\usepackage{latexsym}
\usepackage{url}
\usepackage{graphicx}
\usepackage{epsf}
\usepackage{amsmath}
\usepackage{cite}

\newtheorem{thm}{Theorem}[section]

\theoremstyle{remark} 
\newtheorem{example}[thm]{Example}

\newcommand{\lap}{\bigtriangleup}
\newcommand{\grad}{\bigtriangledown}

\newcommand{\floor}[1]{\lfloor #1 \rfloor}
\newcommand{\Rn}{\mathbb{R}^n}
\newcommand{\nsphere}{S^{n-1}}

\newcommand{\dcube}{\ncubep{d}}
\newcommand{\ncubep}[1]{\mathcal{C}^{#1}}

\newcommand{\RR}{\mathbb{R}}
\newcommand{\half}{\frac{1}{2}}
\newcommand{\LtwoNorm}[1]{\| #1 \|_{L_2}}

\newcommand{\LipNorm}[1]{\| #1 \|_{Lip}}

\newcommand{\spaceo}{\hspace{2 mm}}

\newcommand{\Ell}{\mathcal{E}}

\newcommand{\tridva}{\frac{3}{2}}

\newcommand{\setsep}{ \spaceo | \spaceo}
\newcommand{\Bsetsep}{ \spaceo \Big | \spaceo}

\newcommand{\combchoose}[2]{\left( \begin{array}{l} {#1} \\ {#2} \end{array} \right)}

\begin{document}

\title{On Projections of Metric Spaces}
\date{}
\author{ Mark Kozdoba }
\maketitle

\begin{abstract}
Let $X$ be a metric space and let $\mu$ be a probability measure on it. 
Consider a Lipschitz map $T: X \rightarrow \Rn$, with Lipschitz constant $\leq 1$. 
Then one can ask whether the image $TX$ can have large projections on many directions. 
For a large class of spaces $X$, we show that there are directions $\phi \in \nsphere$ on which the projection of the image $TX$ is small on the average, with bounds 
depending on the dimension $n$ and the eigenvalues of the Laplacian on $X$. 
\end{abstract}

\section{Introduction}
Let $(X,m)$ be a metric space and let $\mu$ be a probability measure on $X$. Consider a 
Lipschitz map $T:X \rightarrow \Rn$, with $\LipNorm{T}\leq 1$, where $\Rn$ is taken with the standard inner product $<\cdot,\cdot>$ and the corresponding Euclidean norm $|\cdot|$. 

Define another (semi) norm on $\Rn$ by setting for each $\theta \in \Rn$, 
\[
   \|\theta\|_{L_2} := \left( \int <Tx,\theta>^2 d\mu(x) \right)^{\half}.
\]

This norm is known as the covariance structure of the push-forward measure $T\mu$, and we shall regard it as measuring the size of the projection of the image 
of $X$ onto the direction $\theta$. It is then natural to ask, for a given space $X$, whether an image of $X$ can have big projections on many directions.  For example, take $X$ to be the 
unit interval $[0,1]$, with Lebesgue measure, and let $T:[0,1] \rightarrow \Rn$ be a Lipschitz map (here and in what follows, Lipschitz will mean ``having Lipschitz constant $\leq 1$").  By considering a few pictures, it 
is clear that if $n$ is large then there should be a direction with a small projection and we 
are interested in quantifying this phenomenon. It turns out that this particular situation was 
investigated already by Gohberg and Krein, in a different formulation and context (related to 
 spectral properties of smooth kernels), see the book \cite{GK}, Chapter III.10.3 . 
\begin{thm}[\cite{GK}]
\label{curves_GK} There is an absolute constant $c>0$, such
that for every $n>1$ and every $T:[0,1] \rightarrow \Rn$ with $\LipNorm{T}\leq 1$, there is 
$\theta \in \nsphere$ such that $\LtwoNorm{T} \leq c \cdot n^{-\tridva}$.
\end{thm}

Our objective is to provide a similar result for general metric spaces with measure on which an appropriate notion 
of a derivative can be defined. However, for the simplicity of presentation, 
we shall restrict our discussion to finite spaces $X$ which are graphs with the shortest 
path metric. The statement and the proof of the following result can, for instance, can be repeated with straightforward modifications in the setting of Riemannian manifolds. 

Let $X$ be a finite set and let $E\subset X \times X$ be a set of edges such that 
$(X,E)$ is a connected undirected graph without loops. For $x,y \in X$, write $x \sim y$ iff $(x,y) \in E$ and let $d(x)$ denote the degree of $x$. We endow $X$ with the shortest path metric and we set $\mu$ to be the stationary distribution of a simple nearest neighbour 
random walk on $(X,E)$, $\mu(x) = \frac{d(x)}{2|E|}$. 
Denote by $L(X)$ the space of real-valued functions on $X$. The Laplacian on 
$L(X)$ is defined by 
\[
 \lap(f)(x) = 2 \left( f(x) - \frac{1}{deg(x)} \sum_{y \sim x} f(y) \right).
\]
As is well known, this is a self-adjoint non-negative operator, with a one dimensional 
kernel consisting of constant functions. We denote by $\{\lambda_i\}_{i=1}^{|X|-1}$ the sequence of non-zero eigenvalues of $\lap$ in \textit{non-decreasing} order, including multiplicities.

\begin{thm} 
\label{main_thm}
There is an absolute constant $c>0$ such that for every graph $X$ as above, every 
$n>1$ and every Lipschitz map $T: X \rightarrow \Rn$, there is a direction 
$\theta \in \nsphere$ such that $\LtwoNorm{\theta} \leq c \cdot n^{-\half} \lambda^{-\half}_{\floor{n/2}-1}$.
\end{thm}

\begin{example}[Discrete Space] 
  Let $(X,m)$ be a metric space such that $m(x,y)=1$ for all $x\neq y$, and let $\mu$ 
  be the uniform probability on $X$. Then for every Lipschitz $T:X \rightarrow \Rn$ there is 
  $\theta \in \Rn$ such that 
  \begin{equation}
    \label{discr_space_bound}
     \LtwoNorm{\theta} \leq \frac{c}{\sqrt{n}}. 
  \end{equation}
  Note that, 
  interestingly enough, the size of the space, $|X|$, does not appear in this bound 
  (except that, of course, for $n>|X|$ we always have $\theta \in \Rn$ with 
  $\LtwoNorm{\theta}=0$). That is, one can not increase the minimal projection size of $X$ in, 
  say, 
  $\RR^{20}$, by adding more points. To prove (\ref{discr_space_bound}), note that 
  $(X,m)$ corresponds to a clique graph, and the Laplacian has a single non-zero eigenvalue, 
  which is of the order of a constant and has multiplicity $|X|-1$. 
\end{example}

\begin{example}[Combinatorial Cube] 
Here $X$ is the set $\dcube = \{0,1\}^d$, with the Hamming metric
\[m(x,y)=|\{i \in \{1,...,d\} \setsep x_i\neq y_i \}|\]
where  $x=x_1...x_d,y=y_1...y_d\in \dcube$ and $\mu$ is again taken to be the uniform probability measure. The non-zero eigenvalues of the Laplacian on $X$ are $\frac{4k}{d}$ with 
multiplicity $\combchoose{n}{k}$, $k=1,...,d$ (See, ex. \cite{Diac}). The $n$-th smallest eigenvalue is therefore 
of the order $\frac{\log(n)}{d}$ and we obtain that for every Lipschitz $T:X \rightarrow \Rn$, 
there is $\theta \in \nsphere$ such that 
\[
    \LtwoNorm{\theta} \leq c \cdot \frac{\sqrt{d}}{\sqrt{n} \cdot \sqrt{\log(n)}}
\]
\end{example}

Our approach to Theorem \ref{main_thm} extends the argument in \cite{GK} and, naturally,
 Theorem \ref{curves_GK} can also be seen as a consequence of the principle of Theorem \ref{main_thm}.  Indeed, the eigenvalues of the Laplacian on $[0,1]$ satisfy,  
up to multiplicative constants, that $\lambda_n \approx n^2$ (as can be seen by double differentiating sines and cosines). Hence the decay rate of $n^{-\half}\cdot (n^{2})^{-\half} = n^{-\tridva}$ in Theorem \ref{curves_GK}.

\section{Proof}
Fix a graph $(X,E)$ with the stationary measure $\mu$ on it. 
Let $T : X \rightarrow \Rn$ be a Lipschitz map. It will be convenient to introduce 
an additional assumption, that 
\begin{equation}
\label{mz_cond}
\int T(x) d\mu = 0. 
\end{equation}
With conditions of Theorem \ref{main_thm} and this assumption, we will 
show that there is $\theta \in \nsphere$ such that 
\begin{equation}
\label{mz_res}
    \LtwoNorm{\theta} = c \cdot n^{-\half} \lambda^{-\half}_{\floor{n/2}}.
\end{equation}
This implies the result for arbitrary Lipschitz $T :X \rightarrow \Rn$. Indeed, assume that 
$v:= \int Tx d\mu \neq 0$. Denote by $P$ the orthogonal projection onto the $n-1$ 
dimensional space orthogonal to $v$. Then the composition $P \circ T$ is a Lipschitz map 
into that space and $\int (P\circ T)(x) d\mu = 0 $, so (\ref{mz_res}) can be applied in dimension $n-1$.

The notion of a gradient on a graph is folklore, although not particularly frequent 
in the literature. We thus recall the definition. 

Define an inner product on $L(X)$ by
\[
     [f,g]_X = \int f(x)g(x) d\mu(x).
\]
Denote by $L(E)$ the set of real valued functions on the set of edges, $E$, and let 
$\nu$ be the uniform probability measure on $E$. We equip $L(E)$ with the inner product 
\[
     [f,g]_E = \int f(e) g(e) d\nu. 
\]
Fix an arbitrary orientation on $E$, i.e. 
for each edge $e=\{x,y\}$ choose in an arbitrary way an enumeration $v_e^1,v_e^2$ of the 
vertices of the edge. The gradient operator on $X$ is defined by 
$\grad : L(X) \rightarrow L(E)$,
\[
    (\grad f)(e) = f(v_e^1) - f(v_e^2).
\]
One can verify by direct computation that the usual relation $\lap = \grad^{*} \circ \grad$ holds, where $\grad^{*}$ is the Hilbert space adjoint of $\grad$. Note that the Laplacian does not depend on the orientation on $E$. 

Denote by $L_0(X)$ the subspace of $L(X)$ that is orthogonal to the constant functions. 
Since $Ker \lap$ is the space of the constant functions, the Laplacian is invertible on $L_0(X)$ and hence the gradient is also an invertible  
operator from $L_0(X)$ onto its image. The inverse of the gradient will be of importance in what follows and we denote it by $\Gamma : Im(\grad) \rightarrow L_0(X)$.

We also recall the notion of singular values of an operator. If $V$ and $W$ are (say, finite dimensional ) Hilbert spaces, and $A:V \rightarrow W$ is a linear operator, then 
$A^{*}A$ is a non-negative self-adjoint operator. Let $\{\lambda_i(A^{*}A)\}_{i=1}^{dim V}$ be the eigenvalues of $A^*A$, in non-increasing order, with multiplicities. Then singular values of the operator $A$ are the non-increasing sequence $s_i(A) = \sqrt{\lambda_i(A^{*}A)}$. 

The singular values of $A$ have a simple geometric interpretation. If $B_2(V)$ denotes 
the unit ball of $V$, then the image $A(B_2(V))$ is an ellipsoid in $W$ and $s_i(A)$ are 
precisely the lengths of the principal axes of this ellipsoid.  

The singular values satisfy $s_i(A) = s_i(A^{*})$ and singular values of a composition 
can be bounded by the following inequality due to Ky Fan. Let $A:U \rightarrow W$ and 
$B: W \rightarrow V$ be two operators on Hilbert spaces. Then, for every $i,j \geq 1$, 
\begin{equation}
\label{ky_fan}
s_{i+j-1}(BA) \leq s_i(A) s_j(B) 
\end{equation}

Finally, the Hilbert-Schmidt norm of an operator is defined by 
\[
   \|A\|_{HS} = \sqrt{tr A^{*} A } = \left(\sum_{i} s_i^2(A)\right)^{\half}. 
\]
Note that since $s_i$ is a non-increasing sequence, 
\begin{equation}
\label{hs_s_i_bound}
  s_i(A) \leq \frac{\|A\|_{HS}}{\sqrt{i}}
\end{equation}
for all $i$. 

More details on singular values can be found, for instance, in \cite{Bha} or \cite{GK}.

\begin{proof}[Proof of Theorem \ref{main_thm}]
As mentioned above, we assume that condition (\ref{mz_cond}) holds. 
Let $D_T$ denote the gradient of the map $T$, i.e. 
\[
D_T(e) = T(v_e^1) - T(v_e^2).
\] 
Since $T$ is Lipschitz, $D_T$ is bounded, i.e. $|D_T(e)|\leq 1$ for all $e\in E$. 

For every $\theta \in \Rn$ consider the function on $X$, $f_{\theta}(x) = <\theta, Tx>$. 
By (\ref{mz_cond}), $f_{\theta} \in L_0(E)$ and clearly  $(\grad f_{\theta}) (e)= <\theta, D_T(e)>$ and 
\[
f_{\theta} = \Gamma (<\theta, D_T>).
\]

Let 
\[
    B_{L_2}(X) = \Big\{f:X \rightarrow \RR \Bsetsep [f,f]_X\leq 1 \Big\}
\]
denote the unit ball in $L(X)$ and write $\LtwoNorm{\theta}$ in a dual form:
\[
   \LtwoNorm{\theta} = sup_{g \in B_2(X)} \int g(x) <\theta, Tx> d\mu .
\]

Then 
\begin{eqnarray*}
  \int g(x) <\theta, Tx> d\mu  = [g, f_{\theta}]_{X} = 
   [g, \Gamma \circ \grad f_{\theta}]_X = [\Gamma^{*} g, \grad f_{\theta}]_E 
\end{eqnarray*}

Next, write 
\begin{eqnarray*}
 [\Gamma^{*} g, \grad f_{\theta}]_E  = 
 \int (\Gamma^{*} g)(e) \cdot <\theta,D_T(e)> d\nu(e) = 
 <\theta, \int (\Gamma^{*} g)(e) \cdot D_T(e) d\nu(e) > .
\end{eqnarray*}

Denote by $\widetilde{D}: L(E) \rightarrow \Rn$ the operator that acts by 
\[\widetilde{D} u = \int u(e) \cdot D_T(e) d\nu\]. 

With this notation, 
\begin{equation}
\label{ell_intro}
   \LtwoNorm{\theta} = sup_{g\in B_2} <\theta, \widetilde{D} \circ \Gamma^{*} g> = 
   \sup_{\phi \in \Ell} <\theta,\phi> 
\end{equation}
where $\Ell = \widetilde{D} \circ \Gamma^{*} B_2$ 
is the dual (in $\Rn$) ellipsoid of the $\LtwoNorm{\cdot}$ norm.

Our aim is to bound the quantity 
\[
    \inf_{\theta \in \nsphere} \LtwoNorm{\theta}.
\]
By (\ref{ell_intro}), this quantity equals the length of the smallest 
principal axe of the ellipsoid $\Ell$. Since the lengths of the principal axes are 
the singular values of $\widetilde{D} \circ \Gamma^{*}$, we bound 
the singular values of this operator. 

The singular values of $\Gamma^{*}$ are given, $s_i(\Gamma^{*}) = \lambda_i^{-\half}$, where 
$\lambda_i$  are the non-zero eigenvalues of the Laplacian (in non-decreasing order, so that 
$s_i(\Gamma^{*})$ do not increase).  To bound the singular values of $\widetilde{D}$, we show that  
\begin{equation}
\label{D_T_HS}
  \|\widetilde{D}\|_{HS}\leq 1.
\end{equation}
Indeed, one readily verifies that 
\[
   \widetilde{D} \circ \widetilde{D}^{*} (\theta) = \int <\theta,D_T(e)>D_T(e) d\nu.  
\]
For fixed $e\in E$, the trace of an operator $\theta \mapsto <\theta,D_T(e)>D_T(e)$ is $|D_T(e)|^2\leq 1$, implying
(\ref{D_T_HS}).  Now, by (\ref{hs_s_i_bound}), $s_i(\widetilde{D}) \leq \frac{1}{\sqrt{i}}$ and an application of (\ref{ky_fan}) (with $i=j=n/2$) completes the proof. 
\end{proof}

\end{document}